\documentclass{article}

\usepackage{amsmath,amssymb,amsthm}
\usepackage{tikz}
\usepackage{xcolor}
\usepackage{float}
\usepackage{array}
\usepackage{amsfonts}
\usepackage{bold-extra}
\usepackage{multirow}

\definecolor{webgreen}{rgb}{0,.5,0}
\definecolor{webbrown}{rgb}{.6,0,0}
\usepackage[colorlinks=true,
linkcolor=webgreen,
filecolor=webbrown,
citecolor=webgreen]{hyperref}
\usepackage[capitalise,noabbrev]{cleveref}
\usepackage{longtable}

\theoremstyle{definition}
\newtheorem{definition}{Definition}[section]

\newtheorem{observation}[definition]{Observation}

\newtheorem{example}[definition]{Example}

\theoremstyle{plain}
\newtheorem{theorem}[definition]{Theorem}
\newtheorem{corollary}[definition]{Corollary}
\newtheorem{lemma}[definition]{Lemma}
\newtheorem{conjecture}[definition]{Conjecture}

\theoremstyle{remark}

\newcommand{\snort}[0]{{\textsc{Snort}} }
\newcommand{\Snort}[0]{{\textsc{Snort}}}
\newcommand{\vertex}[2]{{\node[main] (#2) at #1 {$#2$};}}

\title{Degrees are Useless in \snort When Measuring Temperature}
\author{}
\date{}

\begin{document}

\begin{center}
  \vskip 1cm{\LARGE\bf
      Degrees are Useless in \snort When Measuring Temperature
    }
  \vskip 1cm
  Svenja Huntemann\footnote{This author's research is supported in part by the Natural Sciences and Engineering Research Council of Canada grant 2022-04273.
  \textit{MSC subject class:} 91A46}\\
  Department of Mathematics and Statistics\\
  Mount Saint Vincent University\\
  Halifax, NS~~B3M~2J6\\
  Canada\\
  \href{mailto:svenja.huntemann@msvu.ca}{\tt svenja.huntemann@msvu.ca} \\
  \ \\
  Tomasz Maciosowski \\
  Concordia University of Edmonton\\
  Canada\\
  \href{mailto:t.maciosowski@gmail.com}{\tt t.maciosowski@gmail.com} \\
\end{center}

\vskip .2 in

\begin{abstract}
  \snort is a two-player game played on a simple graph in which players alternately colour a vertex such that they do not colour adjacent to their opponents' vertex.
  In combinatorial game theory, the temperature of a position is a measure of the urgency of moving first.
  It is known that the temperature of \snort in general is infinite ($K_{1,n}$ has temperature $n$).
  We show that the temperature in addition can be infinitely larger than the degree of the board being played on.
  We do so by constructing a family of positions in which the temperature grows twice as fast as the degree of the board.
\end{abstract}

\section{Introduction}

In combinatorial game theory, the temperature of a component of a game represents how much a player will gain by going first in that component, assuming optimal play by both players. Positions in which the temperature is positive, i.e.\ those where the players want to move first, are called hot. In most cases, players want to move in a component with higher temperature. Thus, game players often search for positions of high temperature, whether actively, as for a lot of computer players, or implicitly, as many good human players do.

If a computer player randomly checks moves to decide which one to play, we will need to know when a ``good enough'' position has been found, i.e.\ a position which has temperature equal to or close to the maximum possible. Thus a question of interest is what the maximum temperature (called the boiling point) for all possible positions or a set of positions is. A recent survey by Berlekamp \cite{Berlekamp2019} discusses the temperature of many different games, including some known bounds.

The game \snort is played on any finite simple graph.
The two players, which we call Left and Right, alternately colour vertices, with Left colouring blue and Right in red.
Two adjacent vertices cannot be coloured differently and the game ends once the current player cannot make any move.
In that case, the current player loses.

For example, in \cref{fig:example-snort-game} a possible sequence of moves is shown where Left wins since Right on his turn cannot make a move.
The moves shown may not be optimal play.

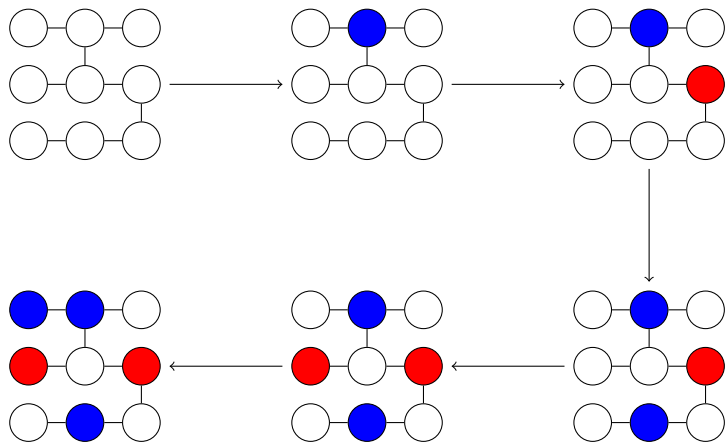
\begin{figure}[H]
  \begin{center}
    \begin{tikzpicture}[vertex/.style={circle, draw, minimum size=5mm, font=\scriptsize},scale=0.75]
      \begin{scope}
        \node[vertex] (1) at (0,0) {};
        \node[vertex] (2) at (1,0) {} edge (1);
        \node[vertex] (3) at (2,0) {} edge (2);
        \node[vertex] (4) at (0,1) {};
        \node[vertex] (5) at (1,1) {} edge (4);
        \node[vertex] (6) at (2,1) {} edge (5);
        \node[vertex] (7) at (0,2) {};
        \node[vertex] (8) at (1,2) {} edge (7);
        \node[vertex] (9) at (2,2) {} edge (8);
        \draw (3) edge (6);
        \draw (5) edge (8);
      \end{scope}
      \draw[->] (2.5,1)--(4.5,1);
      \begin{scope}[shift={(5,0)}]
        \node[vertex] (1) at (0,0) {};
        \node[vertex] (2) at (1,0) {} edge (1);
        \node[vertex] (3) at (2,0) {} edge (2);
        \node[vertex] (4) at (0,1) {};
        \node[vertex] (5) at (1,1) {} edge (4);
        \node[vertex] (6) at (2,1) {} edge (5);
        \node[vertex] (7) at (0,2) {};
        \node[vertex,fill=blue] (8) at (1,2) {} edge (7);
        \node[vertex] (9) at (2,2) {} edge (8);
        \draw (3) edge (6);
        \draw (5) edge (8);
      \end{scope}
      \draw[->] (7.5,1)--(9.5,1);
      \begin{scope}[shift={(10,0)}]
        \node[vertex] (1) at (0,0) {};
        \node[vertex] (2) at (1,0) {} edge (1);
        \node[vertex] (3) at (2,0) {} edge (2);
        \node[vertex] (4) at (0,1) {};
        \node[vertex] (5) at (1,1) {} edge (4);
        \node[vertex,fill=red] (6) at (2,1) {} edge (5);
        \node[vertex] (7) at (0,2) {};
        \node[vertex,fill=blue] (8) at (1,2) {} edge (7);
        \node[vertex] (9) at (2,2) {} edge (8);
        \draw (3) edge (6);
        \draw (5) edge (8);
      \end{scope}
      \draw[->] (11,-0.5)--(11,-2.5);
      \begin{scope}[shift={(10,-5)}]
        \node[vertex] (1) at (0,0) {};
        \node[vertex,fill=blue] (2) at (1,0) {} edge (1);
        \node[vertex] (3) at (2,0) {} edge (2);
        \node[vertex] (4) at (0,1) {};
        \node[vertex] (5) at (1,1) {} edge (4);
        \node[vertex,fill=red] (6) at (2,1) {} edge (5);
        \node[vertex] (7) at (0,2) {};
        \node[vertex,fill=blue] (8) at (1,2) {} edge (7);
        \node[vertex] (9) at (2,2) {} edge (8);
        \draw (3) edge (6);
        \draw (5) edge (8);
      \end{scope}
      \draw[->] (9.5,-4) -- (7.5,-4);
      \begin{scope}[shift={(5,-5)}]
        \node[vertex] (1) at (0,0) {};
        \node[vertex,fill=blue] (2) at (1,0) {} edge (1);
        \node[vertex] (3) at (2,0) {} edge (2);
        \node[vertex,fill=red] (4) at (0,1) {};
        \node[vertex] (5) at (1,1) {} edge (4);
        \node[vertex,fill=red] (6) at (2,1) {} edge (5);
        \node[vertex] (7) at (0,2) {};
        \node[vertex,fill=blue] (8) at (1,2) {} edge (7);
        \node[vertex] (9) at (2,2) {} edge (8);
        \draw (3) edge (6);
        \draw (5) edge (8);
      \end{scope}
      \draw[->] (4.5,-4) -- (2.5,-4);
      \begin{scope}[shift={(0,-5)}]
        \node[vertex] (1) at (0,0) {};
        \node[vertex,fill=blue] (2) at (1,0) {} edge (1);
        \node[vertex] (3) at (2,0) {} edge (2);
        \node[vertex,fill=red] (4) at (0,1) {};
        \node[vertex] (5) at (1,1) {} edge (4);
        \node[vertex,fill=red] (6) at (2,1) {} edge (5);
        \node[vertex,fill=blue] (7) at (0,2) {};
        \node[vertex,fill=blue] (8) at (1,2) {} edge (7);
        \node[vertex] (9) at (2,2) {} edge (8);
        \draw (3) edge (6);
        \draw (5) edge (8);
      \end{scope}
    \end{tikzpicture}
  \end{center}
  \caption{Example \snort Game}
  \label{fig:example-snort-game}
\end{figure}

\snort is named after Simon Norton, who originally introduced it (see \cite{BerlekampCG2004}).
When played on grids, the game is also known as \textsc{Cats and Dogs} and is played by school children in Portugal as part of a national games tournament.

The game is interesting to play since most moves will both ``reserve'' vertices for the current player, and may also ``block'' vertices the opponent previously reserved.
\snort also naturally breaks into smaller, independent components as the game progresses and we can analyze the components to learn more about the overall game.

Since its introduction, \snort has been studied from various perspectives. Winning Ways \cite{BerlekampCG2004} contains a dictionary of values when playing on paths up to six vertices. It was shown that \snort is PSPACE-complete in \cite{Schaefer1978}. The authors in \cite{BrownCHMMNS2019} and \cite{HuntemannN2022} enumerated positions in \snort on a variety of boards using generating polynomials. The outcome class when playing on a variety of uncoloured graphs was studied in \cite{Arroyo1998} and \cite{Kakihara2010}. And \cite{HuntemannNP2021} gave a first upper bound on the temperature when playing \snort on paths, although the conjectured bound is significantly lower.

When playing \snort on $K_{1,4}$ as a component, the optimal move for either player is to move in the centre, thereby reserving the other four vertices for themselves, without the opponent being able to prevent future moves in them.
Since this results in four additional moves the current player can make at any time (for example if they wanted to force their opponent to make the next move in another component), we say that $K_{1,4}$ has temperature $4$.

It is known in general that playing \snort on $K_{1,n}$ has temperature $n$, thus the temperature in \snort is unbounded.

The positions analyzed in \cite{BerlekampCG2004} all satisfy $t(G)\leq\deg(G)$, and it was conjectured (see for example \cite{HuntemannPhD}) that this bound holds in general.
At the Virtual Combinatorial Games Workshop in 2020 a position was found that has $t(G)>\deg(G)$.
Several more such positions were found at the Combinatorial Game Theory Workshop in St.\ John's, Canada in 2022, all of which satisfied $t(G)-\deg(G)\leq 2$.
One example of such a position is shown in \cref{fig:example-initial-position}.

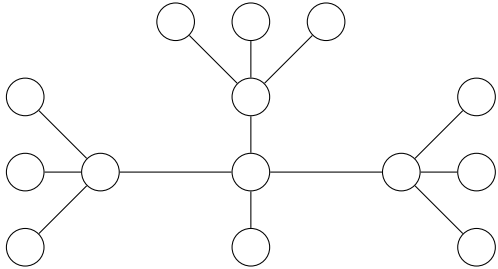
\begin{figure}[H]
  \begin{center}
    \begin{tikzpicture}[vertex/.style={circle, draw, minimum size=5mm, font=\scriptsize}]
      \node[vertex] (1) at (0,0) {};
      \node[vertex] (2) at (2,0) {} edge (1);
      \node[vertex] (3) at (4,0) {} edge (2);
      \node[vertex] (4) at (-1, 1) {} edge (1);
      \node[vertex] (5) at (-1, 0) {} edge (1);
      \node[vertex] (6) at (-1, -1) {} edge (1);
      \node[vertex] (7) at (2, 1) {} edge (2);
      \node[vertex] (8) at (2, -1) {} edge (2);
      \node[vertex] (9) at (1, 2) {} edge (7);
      \node[vertex] (10) at (2, 2) {} edge (7);
      \node[vertex] (11) at (3, 2) {} edge (7);
      \node[vertex] (12) at (5, 1) {} edge (3);
      \node[vertex] (13) at (5, 0) {} edge (3);
      \node[vertex] (14) at (5, -1) {} edge (3);
    \end{tikzpicture}
  \end{center}
  \caption{A \snort position with $t(G) - \deg(G) = \frac{3}{2}$}
  \label{fig:example-initial-position}
\end{figure}

We will show in this paper that the difference $t(G)-\deg(G)$ is unbounded.
We used a genetic algorithm with fitness function $t(G)-\deg(G)$, with the initial generation being the positions from the 2022 workshop, to find several positions with larger differences.
We then generalized these to a family of graphs where the temperature grows twice as fast as the degree.

\section{Game Values and Temperatures: Background}

In this section, we will give an informal introduction to game values and temperatures.
The reader familiar with combinatorial game theory can skip this section.
For more details, see for example \cite{siegel2013combinatorial}.

We use the notation $G = \{G^\mathcal{L} \mid G^\mathcal{R}\}$ to describe a game position where Left can move to any position in the set $G^\mathcal{L}$ and Right to any position in $G^\mathcal{L}$.
The game in which neither player has any options is denoted $0 = \{\; \mid \; \}$ and any game in which the first player to move loses is equal to $0$.
The game in which both players can move to 0 is denoted $* = \{0 \mid 0\}$.
The \emph{positive integers} are recursively constructed as $n = \{ n-1 \mid \;\}$, thus are games in which Left has $n$ consecutive moves and Right has none.
The \emph{negative} of a game is the game in which the roles for the two players are reversed, i.e.\ $-G = \{- G^\mathcal{R} \mid -G^\mathcal{L} \}$. In \Snort, this is equivalent to swapping the colours of any vertices already claimed.
The game in which Right's options are all negatives of Left's options is denoted $\pm\{G^\mathcal{L}\}=\{G^\mathcal{L}\mid-G^\mathcal{L}\}$, e.g.\ $\pm 1=\{1\mid -1\}$.

A game can be simplified into an equivalent game by removing \emph{dominated options}, those which a player would never play since another option is strictly better, and replacing \emph{reversible options} (which we do not need in this paper).
This simplified game is called the \emph{value} or the \emph{canonical form} of the game.

\medskip

The temperature of a game can be used to gauge in which disjoint component of a game to move, although there are some cases in which the best move is not always in the component with the highest temperature.
The temperature is a dyadic rational in the interval $[-1,\infty)$, with positive temperatures assigned to those games where a player gains by moving first, such as \snort on $K_{1,n}$.
In contrast, an integer has temperature $-1$ because the player removes one of their guaranteed moves when playing in it.

We can get the temperature of a game by constructing a thermograph.
We will describe how to construct the thermograph recursively from the options.
Note that the axes in a thermograph use the following convention: The vertical axis is for the independent variable, the amount by which the game is ``cooled'', equivalent to a penalty paid by the players to move first.
The horizontal axis is oriented with positive values on the left.
The left side of a thermograph represents the worth of the component to Left when paying the corresponding penalty, and similarly for the right side.
The point at which the two sides meet is the temperature of the game, after which a vertical mast is drawn.

For the simplest case, if the component is an integer, the thermograph is only a mast at that value and its temperature is $-1$ (see \cref{fig:thermograph-integer} for the thermograph of $G=3$).

\begin{figure}[H]
  \centering
  \begin{tikzpicture}
    \draw[<-] (0,0) -- (6,0);
    \draw[very thick, ->] (2, -1) -- (2, 3.5);
    \draw[->] (5, -1) -- (5, 4);
    \draw (4.8, -0.1) node[anchor=north] {$0$};
    \draw (1.8, -0.1) node[anchor=north] {$3$};
    \draw[dotted] (2,-1) -- (5,-1) node[anchor=west] {$t = -1$};
  \end{tikzpicture}
  \caption{Thermograph of $G = 3$}
  \label{fig:thermograph-integer}
\end{figure}
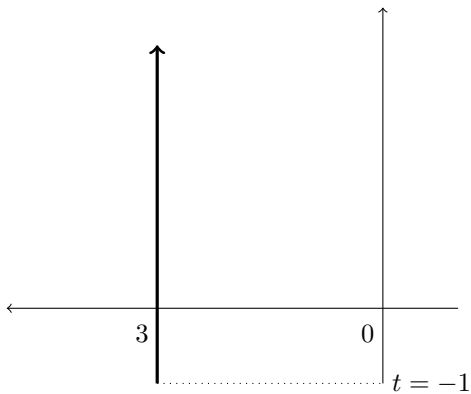

To construct thermographs of a position in which both players have one option, we use the thermographs of the options.
We will draw the thermographs of the options in the same diagram using dashed lines.
We use the right side of the thermograph of the Left option and sheer it by $45^\circ$ to the right, and the left side of the thermograph of the Right option and sheer it to the left.
At the intersection point the vertical mast starts.
See \cref{fig:thermograph-two-integers} for the thermograph of $G=\{2\mid -1\}$ and \cref{fig:thermograph-complex} for the thermograph of $G=\big\{\{2\mid -1\}\mid\{-4\mid -10\}\big\}$, for the latter of which the thermographs of the options were found recursively.

\begin{figure}[H]
  \centering
  \begin{tikzpicture}
    \draw[<-] (-4,0) -- (3,0);

    \draw[->] (0, -1) -- (0, 4);
    \draw (-0.2, -0.1) node[anchor=north] {$0$};

    \draw[dashed,->] (-2, -1) -- (-2, 3.5);
    \draw (-1.8, -0.1) node[anchor=north] {$2$};

    \draw[dashed,->] (1, -1) -- (1, 3.5);
    \draw (0.75, -0.1) node[anchor=north] {$-1$};

    \draw[very thick] (-3, -1) -- (-2, 0) -- (-1/2, 3/2);
    \draw[very thick] (2, -1) -- (1, 0) -- (-1/2, 3/2);
    \draw[very thick, ->] (-1/2, 3/2) -- (-1/2, 3.5);

    \draw[dotted] (-1/2, 3/2) -- (0, 3/2) node[anchor=west] {$t = \frac{3}{2}$};
  \end{tikzpicture}
  \caption{Thermograph of $G = \{2 \mid -1 \}$}
  \label{fig:thermograph-two-integers}
\end{figure}
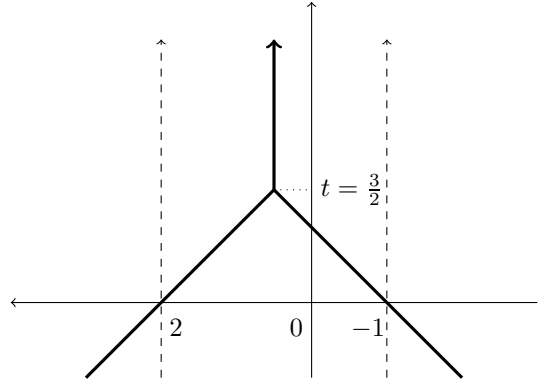

\begin{figure}[H]
  \centering
  \begin{tikzpicture}[scale=0.75]
    \draw[<-] (-3,0) -- (11,0);

    \draw[->] (0, -1) -- (0, 5);
    \draw (-0.2, -0.1) node[anchor=north] {$0$};

    \draw[dashed] (-3, -1) -- (-2, 0) -- (-1/2, 3/2);
    \draw[dashed] (2, -1) -- (1, 0) -- (-1/2, 3/2);
    \draw[dashed, ->] (-1/2, 3/2) -- (-1/2, 4.5);
    \draw (-1.8, -0.1) node[anchor=north] {$2$};
    \draw (0.65, -0.1) node[anchor=north] {$-1$};

    \draw[dashed] (3, -1) -- (4, 0) -- (7, 3);
    \draw[dashed] (11, -1) -- (10, 0) -- (7, 3);
    \draw[dashed, ->] (7, 3) -- (7, 4.5);
    \draw (4, -0.1) node[anchor=north] {$-4$};
    \draw (9.75, -0.1) node[anchor=north] {$-10$};

    \draw[very thick, ->] (1, -1) -- (1, 3/2) -- (13/4, 15/4) -- (13/4, 4.5);
    \draw[very thick] (4, -1) -- (4, 3) -- (13/4, 15/4);
  \end{tikzpicture}
  \caption{Thermograph of $G=\big\{\left\{2 \mid -1\right\} \mid \left\{-4 \mid -10\right\}\big\}$}
  \label{fig:thermograph-complex}
\end{figure}
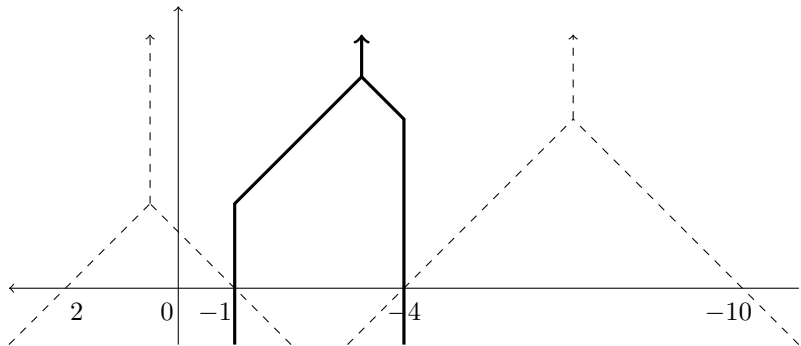

When the players have multiple options, to find the left side, we overlay the thermographs for all Left options and choose the left-most point of the right sides for any value of the independent variable, then again sheer to the right.
And similarly for the right side. As an example, the thermograph of $G=\big\{-1,\{2\mid-2\}\mid -8\big\}$ is shown in \cref{fig:thermograph-multiple-options}.

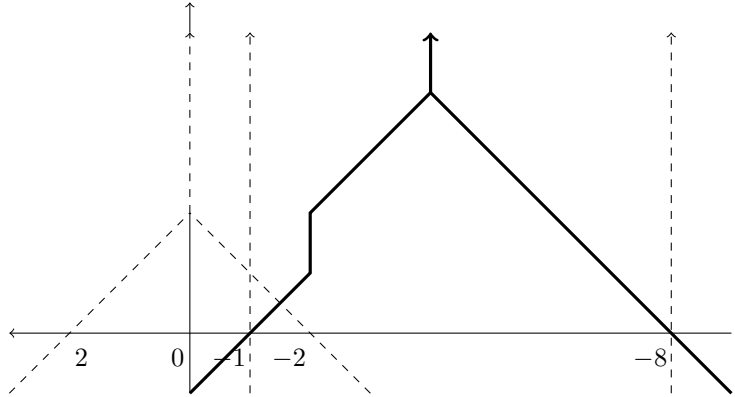
\begin{figure}[H]
  \centering
  \begin{tikzpicture}[scale=0.8]
    \draw[<-] (-3,0) -- (9,0);

    \draw (0, -1) -- (0, 2);
    \draw[->] (0, 5) -- (0, 5.5);
    \draw (-0.2, -0.1) node[anchor=north] {$0$};

    \draw[dashed] (-3, -1) -- (-2, 0) -- (0, 2);
    \draw[dashed] (3, -1) -- (2, 0) -- (0, 2);
    \draw[dashed, ->] (0, 2) -- (0, 5);
    \draw (-1.8, -0.1) node[anchor=north] {$2$};
    \draw (1.65, -0.1) node[anchor=north] {$-2$};

    \draw[dashed,->] (1, -1) -- (1, 5);
    \draw (0.65, -0.1) node[anchor=north] {$-1$};

    \draw[dashed,->] (8, -1) -- (8, 5);
    \draw (7.65, -0.1) node[anchor=north] {$-8$};

    \draw[very thick,->] (0, -1) -- (2, 1) -- (2, 2) -- (4, 4) -- (4, 5);
    \draw[very thick] (9, -1) -- (4, 4);
  \end{tikzpicture}
  \caption{Thermograph of $G=\big\{-1,\{2\mid-2\}\mid -8\big\}$}
  \label{fig:thermograph-multiple-options}
\end{figure}

\section{Preliminary Results}
When analyzing a \snort position, rather than keeping track of the colours of adjacent vertices, we simplify the board by removing the vertex that the player played on and tinting all adjacent vertices in their colour.

For example, when playing on a path of five vertices, if Left plays adjacent to a leaf, then the resulting graph is an isolated vertex tinted blue plus a path of three vertices with one leaf tinted blue (see below).

\begin{center}
  \begin{tikzpicture}[scale=0.6, vertex/.style={circle, draw, minimum size=3mm, font=\scriptsize}]
    \node[vertex] (1) at (0,0) {};
    \node[vertex] (2) at (1,0) {} edge (1);
    \node[vertex] (3) at (2,0) {} edge (2);
    \node[vertex, fill=blue] (4) at (3,0) {} edge (3);
    \node[vertex] (5) at (4,0) {} edge (4);
    \draw[->] (5,0)--(6,0);
    \begin{scope}[shift={(7,0)}]
      \node[vertex] (1a) at (0,0) {};
      \node[vertex] (2a) at (1,0) {} edge (1a);
      \node[vertex, fill=blue!30] (3a) at (2,0) {} edge (2a);
      \node[vertex,fill=blue!30] (5a) at (4,0) {};
    \end{scope}
  \end{tikzpicture}
\end{center}

If a node is tinted in both colours it is also removed from the graph as neither player can play in this vertex anymore.

For example, if Right continues in the game above and plays in the centre of the path, then the previously tinted blue leaf is also tinted red and therefore removed.
Thus, two isolated vertices, one tinted blue and one red, will remain (see below).

\begin{center}
  \begin{tikzpicture}[scale=0.6, vertex/.style={circle, draw, minimum size=3mm, font=\scriptsize}]
    \node[vertex] (1) at (0,0) {};
    \node[vertex, fill=red] (2) at (1,0) {} edge (1);
    \node[vertex, fill=blue!30] (3) at (2,0) {} edge (2);
    \node[vertex, fill=blue!30] (5) at (4,0) {};
    \draw[->] (5,0)--(6,0);
    \begin{scope}[shift={(7,0)}]
      \node[vertex, fill=red!30] (1a) at (0,0) {};
      \node[vertex, fill=blue!30] (5a) at (1,0) {};
    \end{scope}

  \end{tikzpicture}
\end{center}

When a game splits into such components, we say that it is the \emph{disjunctive sum} of the components.





Before moving to our main result, we will discuss the game values when playing \snort on several graphs that commonly occur after splitting into a disjunctive sum.

We will assume that vertices are not tinted unless we specify otherwise.

\begin{observation}\label{lem:isolated}
  The game value of \snort played on $n$ blue-tinted isolated vertices is $n$.
  The game value of $n$ isolated vertices is $0$ if $n$ is even and $*$ otherwise.
\end{observation}

\begin{lemma}\label{lem-n-star-tinted}
  The game value of \snort played on the star $K_{1, n}$, with centre tinted blue, is $\{ n \mid * \}$ if $n$ is even, and $\{ n \mid 0 \}$ if $n$ is odd.
\end{lemma}
\begin{proof}
  Right can only move in a leaf as the center is already tinted blue.
  This move will remove the leaf and the centre creating a position with $n - 1$ isolated vertices, thus its value is $*$ if $n$ is even, and $0$ if $n$ is odd.
  Left can move either in the centre, creating $n$ isolated blue-tinted vertices of value $n$, or in a leaf to the position $K_{1,n - 1}$ with blue-tinted centre.
  The latter is dominated by induction.
  Thus the game value is $\{ n \mid * \}$ if $n$ is even, and $\{ n \mid 0 \}$ if $n$ is odd.
\end{proof}

\begin{lemma}\label{lem-n-star}
  The game value of \snort played on the star $K_{1, n}$ is $\pm n$.
\end{lemma}
\begin{proof}
  There are only two distinct moves up to symmetry --- in the centre or in a leaf.
  Left moving in the center will create $n$ tinted vertices, so the value of this move is $n$.
  Moving in a leaf will remove it and tint the center of the star, thus its value is $\{ n - 1 \mid 0 \}$ if $n$ is even and $\{ n - 1 \mid * \}$ when $n$ is odd, which is dominated by $n$.
  Right's moves are the negatives of these moves as the initial position has no tinted vertices.
  Therefore, the game value is $\pm n$.
\end{proof}

\begin{observation}\label{lem:connected-stars}
  As discussed in \cite{BerlekampCG2004}, the edge between two adjacent vertices tinted the same colour can be removed.
  In particular, two stars with centres tinted the same colour and connected with an edge have the same game value as if the centres were not connected.
\end{observation}

\begin{lemma}\label{lem:connected-stars-diff-colors}
  \snort played on the graph consisting of two copies of $K_{1, n}$ connected at their centres, which are tinted opposite colours, has game value $\pm n$ if $n$ is even and $\pm (n*)$ if $n$ is odd.
\end{lemma}
\begin{proof}
  Up to symmetry, each player has the same three moves, i.e.\ in the centre tinted their colour, in a leaf connected to the centre of their own colour, or in a leaf connected to the centre of their opponent's colour.
  Thus, we will consider only moves for Left.

  Moving in the blue-tinted centre results in a position with $n$ blue-tinted isolated vertices, and $n$ further isolated vertices. So the game value is $n$ if $n$ is even and $n*$ if $n$ is odd.
  We will show that all other moves are dominated.

  If Left were to move in a leaf of the red-tinted centre, Right can respond with a move in a leaf of the blue-tinted center into the position with $2n - 2$ isolated vertices, which has a game value of $0$ and is thus dominated by the option $n$ or $n*$.

  After Left moves in a leaf of the blue-tinted centre, Red can respond in the red-tinted centre, to the position with $n - 1$ isolated vertices and $n$ red-tinted isolated vertices, which has a game value $-n*$ if $n$ is even and $-n$ if $n$ is odd, which is again dominated.

  Therefore the game value of the position is $\pm n$ if $n$ is even and $\pm (n*)$ if $n$ is odd.
\end{proof}

\section{Playing on Caterpillars}

Since the game value of \snort played on the star $K_{1, n}$ is $\pm n$, we know that the temperature for this position is $n$, and thus equal to the degree of the board.
In this section we will show that the temperature can be unboundedly larger than the degree.
We will consider a family of positions played on a particular caterpillar tree (see \cite{harary1973number} for more on caterpillars).

\begin{definition}
  The caterpillar $C(n+1,n,n+1)$ consists of a main path of length 3, whose central vertex has $n$ leaves added, and $n+1$ leaves are added to the side vertices.
  I.e., the vertices are
  \[V=\{a,b,c,a_1,\ldots,a_{n+1},b_1,\ldots,b_n,c_1,\ldots,c_{n+1}\}\]
  and the edge set is
  \[E=\{ab,bc\}\cup\bigcup_{i=1}^{n+1}\{a_ia,c_ic\}\cup\bigcup_{j=1}^n\{b_jb\}.\]
\end{definition}

\begin{example}
  The caterpillars $C(4,3,4)$ and $C(5,4,5)$ are shown in \cref{fig:caterpillar-tree}.
\end{example}

\begin{figure}[H]
  \resizebox{\linewidth}{!}{
    \begin{tikzpicture}[node distance={14mm}, main/.style = {draw, circle, minimum size=8mm}]
      \vertex{(0,0)}{a}
      \vertex{(1,0)}{b}
      \vertex{(2,0)}{c}

      \path (a) edge (b);
      \path (b) edge (c);

      \vertex{(-1.5,1.5)}{a_1}
      \vertex{(-1.5,0.5)}{a_2}
      \vertex{(-1.5,-0.5)}{a_3}
      \vertex{(-1.5,-1.5)}{a_4}

      \path (a) edge (a_1);
      \path (a) edge (a_2);
      \path (a) edge (a_3);
      \path (a) edge (a_4);

      \vertex{(0, -1.5)}{b_1}
      \vertex{(1, -1.5)}{b_2}
      \vertex{(2, -1.5)}{b_3}

      \path (b) edge (b_1);
      \path (b) edge (b_2);
      \path (b) edge (b_3);

      \vertex{(3.5,1.5)}{c_1}
      \vertex{(3.5,0.5)}{c_2}
      \vertex{(3.5,-0.5)}{c_3}
      \vertex{(3.5,-1.5)}{c_4}

      \path (c) edge (c_1);
      \path (c) edge (c_2);
      \path (c) edge (c_3);
      \path (c) edge (c_4);
    \end{tikzpicture}
    \hspace{1cm}

    \begin{tikzpicture}[node distance={14mm}, main/.style = {draw, circle, minimum size=8mm}]
      \vertex{(0,0)}{a}
      \vertex{(1,0)}{b}
      \vertex{(2,0)}{c}

      \path (a) edge (b);
      \path (b) edge (c);

      \vertex{(-1.5,2)}{a_1}
      \vertex{(-1.5,1)}{a_2}
      \vertex{(-1.5,0)}{a_3}
      \vertex{(-1.5,-1)}{a_4}
      \vertex{(-1.5,-2)}{a_5}

      \path (a) edge (a_1);
      \path (a) edge (a_2);
      \path (a) edge (a_3);
      \path (a) edge (a_4);
      \path (a) edge (a_5);

      \vertex{(-0.5, -2)}{b_1}
      \vertex{(0.5, -2)}{b_2}
      \vertex{(1.5, -2)}{b_3}
      \vertex{(2.5, -2)}{b_4}

      \path (b) edge (b_1);
      \path (b) edge (b_2);
      \path (b) edge (b_3);
      \path (b) edge (b_4);

      \vertex{(3.5,2)}{c_1}
      \vertex{(3.5,1)}{c_2}
      \vertex{(3.5,0)}{c_3}
      \vertex{(3.5,-1)}{c_4}
      \vertex{(3.5,-2)}{c_5}

      \path (c) edge (c_1);
      \path (c) edge (c_2);
      \path (c) edge (c_3);
      \path (c) edge (c_4);
      \path (c) edge (c_5);
    \end{tikzpicture}
  }
  \caption{The Caterpillars $C(4,3,4)$ (on the left) and $C(5,4,5)$ (on the right).}
  \label{fig:caterpillar-tree}
\end{figure}
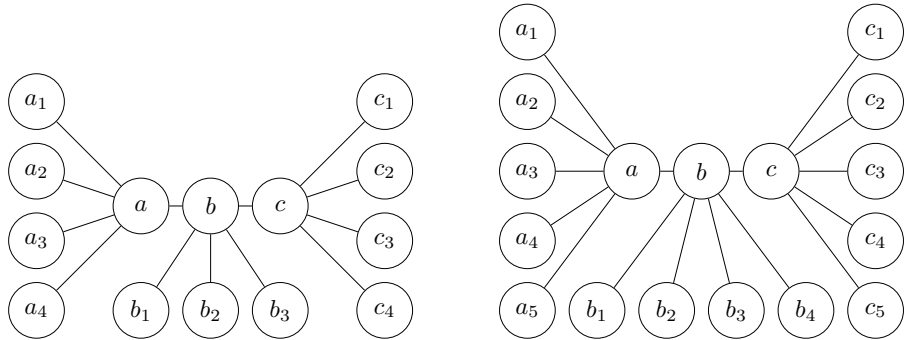

\cref{tab:variousN} contains the canonical forms and temperatures when playing \snort on $C(n+1,n,n+1)$ for small values of $n$, compared to the degree of the caterpillar.
\begin{table}[H]
  \centering
  \begin{tabular}{l|l|l|l|l}
    $n$ & Canonical Form                            & $t(G)$ & $\deg(G)$ & $t(G)-\deg(G)$ \\
    \hline
    $1$ & $\pm \{3*,  \{\{5  \mid 4\}   \mid *\}\}$ & 3      & 3         & 0              \\
    $2$ & $\pm \{5,   \{\{8  \mid 6*\}  \mid 0\}\}$ & 5      & 4         & 1              \\
    $3$ & $\pm \{7*,  \{\{11 \mid 8\}   \mid *\}\}$ & 7      & 5         & 2              \\
    $4$ & $\pm \{9,   \{\{14 \mid 10*\} \mid 0\}\}$ & 9      & 6         & 3              \\
    $5$ & $\pm \{11*, \{\{17 \mid 12\}  \mid *\}\}$ & 11     & 7         & 4              \\
    $6$ & $\pm \{13,  \{\{20 \mid 14*\} \mid 0\}\}$ & 13     & 8         & 5              \\
    $7$ & $\pm \{15*, \{\{23 \mid 16\}  \mid *\}\}$ & 15     & 9         & 6              \\
    $8$ & $\pm \{17,  \{\{26 \mid 18*\} \mid 0\}\}$ & 17     & 10        & 7              \\
  \end{tabular}
  \caption{The canonical form and temperature of \snort played on $C(n+1,n,n+1)$ compared to the degree of the board}
  \label{tab:variousN}
\end{table}

From this data, we generalize the canonical form as follows.

\begin{theorem}\label{thm-canonical}
  The canonical form of \snort played on the caterpillar $C(n+1, n, n+1)$ is
  \[\pm \left\{ 2n + 1 + s(n), \left\{ \left\{ 3n + 2 \mid 2n + 2 + s(n - 1) \right\} \mid s(n) \right\} \right\}\]
  where $s(n)$ is $0$ if $n$ is even and $*$ if $n$ is odd.
\end{theorem}
\begin{proof}
  Since the board is empty, the options for Right are the negative of the Left options.
  Thus it is enough for us to determine the options for Left.

  There are only four possible first moves up to symmetry: Left could move in $b$, move in $a$ or $c$, move in a leaf of $b$, or move in a leaf of $a$ or $c$.

  \medskip

  \textbf{Case 1: Move in $b$}

  First, when Left moves in $b$, the board will split into $n$ blue-tinted isolated vertices (those that were adjacent to $b$) and two stars with blue-tinted centers and $n+1$ leafs each.
  Using \cref{lem:isolated} and \cref{lem-n-star-tinted}, the value of this option is \[n + \{n+1 \mid s(n)\} + \{n+1 \mid s(n)\} = 2n + 1 + s(n).\]

  \medskip

  \textbf{Case 2: Move in $a$ or $c$}

  Now consider when Left moves in $a$.

  If Left moves again in this position, we claim that the only sensible move is in $c$ to the position with $2n+2$ isolated blue-tinted vertices and a star with a blue-tinted center with $n$ leaves, which has a value of \[2n + 2 + \{n \mid s(n - 1)\} = \{ 3n + 2 \mid 2n + 2 + s(n - 1) \}.\] All other moves are worse for Left:
  \begin{itemize}
    \item The move in $b$ is dominated as it leads to a position with $2n + 1$ isolated blue-tinted vertices and a star with a blue-tinted center with $n + 1$ leaves with similar, but less favourable, value $(2n + 1) + \{n + 1 \mid s(n) \} = \{ 3n + 2 \mid 2n + 1 + s(n) \}$.
    \item A move in a leaf of $b$ is dominated by the move in $c$, which can be seen by showing that Left wins in the difference going second. 
    \item A move in a leaf of $c$ gives us $n + 1$ blue-tinted isolated vertices, plus two stars with blue-tinted centres and $n$ leaves with connected centres, which is equal to the centres not being connected by \cref{lem:connected-stars}. Thus the game value is $(n + 1) + \{n \mid s(n - 1) \} + \{n \mid s(n - 1) \} = 2n + 1 + s(n - 1)$, which is dominated.
    \item The isolated blue-tinted vertices do not need to be considered by the number avoidance principle since they have value 1.
  \end{itemize}

  Right's good responding move is in $c$.
  In the resulting position, $b$ is no longer playable and thus removed and we are left with two stars with tinted centre, one blue and one red, with $n+1$ leaves each, and $n$ isolated vertices.
  Since the two stars are negatives of each other, they cancel and the position has value $s(n)$.
  All other moves for Right do not need to be considered:
  \begin{itemize}
    \item A move in a leaf of $b$ removes it and results in $n-1$ empty isolated vertices (the leaves of $b$), a star with empty centre and $n+1$ leaves, and $n+1$ blue-tinted isolated vertices from Left's move in $a$.
          Therefore, the position has value $s(n - 1) \pm(n + 1) + (n + 1)= \{2n + 2 + s(n - 1) \mid s(n - 1) \}$.
    \item A move in a leaf of $c$ leads to two stars - one with blue-tinted, one with red-tinted center, and $n$ leaves each and centres connected by an edge, and $n+1$ blue-tinted isolated vertices from Left's move in $a$.
          By \cref{lem:connected-stars-diff-colors} the game value of this position is $\pm (n + s(n)) + (n + 1) = \{2n + s(n) \mid 1 + s(n)\}$.
  \end{itemize}

  \textbf{Case 3: Move in a leaf of $b$}

  A move in a leaf of $b$ is dominated as Right can respond with a move in $a$ or $c$ to a position consisting of a star with $n + 1$ leaves, $n - 1$ empty isolated vertices, and $n + 1$ red-tinted isolated vertices.
  This has game value \[\pm(n + 1) + s(n - 1) - (n + 1) = \{s(n - 1) \mid -2n - 2 + s(n - 1)\},\] which is dominated as now Left can move only to $s(n - 1) < 3n + 2$.

  \medskip

  \textbf{Case 4: Move in a leaf of $a$ or $c$}

  A move in a leaf of $a$ is dominated as Right can respond with a move in $b$.
  This results in a position with $n$ isolated empty vertices, $n$ isolated red-tinted vertices, and a star with red-tinted centre and $n + 1$ leaves, which has a value of $s(n) - n - \{ n + 1 \mid s(n)\} = \{ -n \mid -2n - 1 + s(n) \}$, which is a win for Right.

  \vspace{0.5cm}

  Thus the only two sensible moves for Left are to $2n + 1 + s(n)$ and $\{ \{ 3n + 2 \mid 2n + 2 + s(n - 1) \} \mid s(n)\}$.
  Since Right's options are the negative of this, so the canonical form is
  \[
    \pm \{ 2n + 1 + s(n), \{ \{ 3n + 2 \mid 2n + 2 + s(n - 1) \} \mid s(n)\} \}. \qedhere
  \]
\end{proof}

\begin{corollary}
  The temperature of \snort played on the caterpillar $C(n+1, n, n+1)$ is $2n+1$.
\end{corollary}
\begin{proof}
  We will show that this is the temperature by constructing a modified version of the thermograph.

  First note that adding $\ast$ to any position will only change a thermograph below the horizontal axis, i.e., only when ``cooling'' by negative values.
  Since we are only interested in the temperature here, and the temperature is positive, it is sufficient if we construct the thermograph above the horizontal axis.
  Instead of using the full canonical form from \cref{thm-canonical}, we will construct the thermograph of the simplified value \[\pm \{ 2n + 1, \{ \{ 3n + 2 \mid 2n + 2 \} \mid 0\} \},\] which is also known as the reduced canonical form.
  In particular, we do not need to consider even and odd values of $n$ separately.

  The thermograph is constructed recursively, with the thermograph of the follower $\{3n+2\mid 2n+2\}$ shown in \cref{fig:thermograph-1}, the one of $\{\{3n+2\mid 2n+2\}\mid 0\}$ in \cref{fig:thermograph-2}, and finally the thermograph of the full game in \cref{fig:thermograph}.

  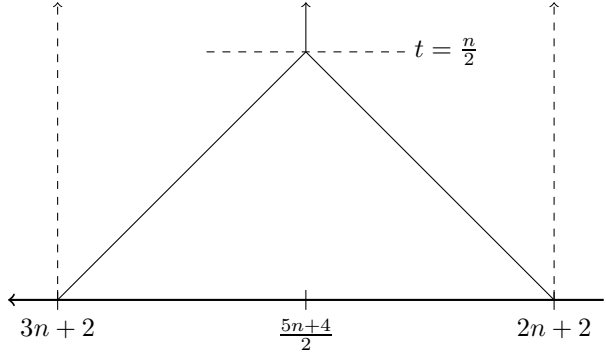
\begin{figure}[H]
    \centering
    \begin{tikzpicture}[scale=0.66]
      \draw[thick,<-] (0,0) -- (12,0); 
      \draw[dashed,->] (1, 0) -- (1, 6);
      \draw[dashed,->] (11, 0) -- (11, 6);

      \draw(1,0) -- (6, 5);
      \draw (1, 0.2) -- (1, -0.2) node[anchor=north] {$3n+2$};

      \draw(11,0) -- (6, 5);
      \draw (11, 0.2) -- (11, -0.2) node[anchor=north] {$2n+2$};

      \draw[->] (6, 5) -- (6, 6);
      \draw (6, 0.2) -- (6, -0.2) node[anchor=north] {$\frac{5n+4}{2}$};
      \draw[dashed] (4,5) -- (8,5) node[anchor=west] {$t = \frac{n}{2}$};

    \end{tikzpicture}
    \caption{Thermograph of $\{ 3n + 2 \mid 2n + 2 \}$}
    \label{fig:thermograph-1}
  \end{figure}

  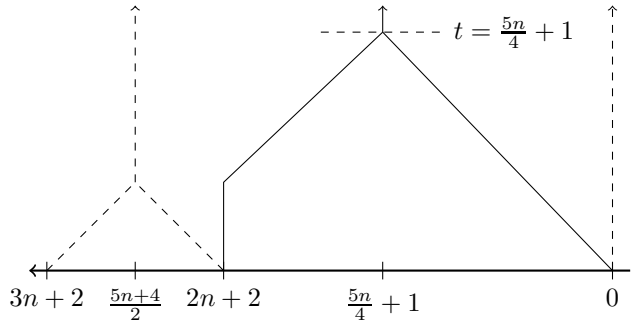
\begin{figure}[H]
    \centering
    \begin{tikzpicture}[scale=0.235]
      \draw[thick,<-] (-33,0) -- (1,0); 

      \draw[dashed] (-32,0) -- (-27, 5);
      \draw (-32, 0.5) -- (-32, -0.5) node[anchor=north] {$3n+2$};

      \draw[dashed] (-22,0) -- (-27, 5);
      \draw (-22, 0.5) -- (-22, -0.5) node[anchor=north] {$2n+2$};

      \draw[dashed,->] (-27,5) -- (-27,15);
      \draw (-27, 0.5) -- (-27, -0.5) node[anchor=north] {$\frac{5n+4}{2}$};


      \draw[dashed,->] (0,0) -- (0,15);
      \draw (0, 0.5) -- (0, -0.5) node[anchor=north] {$0$};

      \draw (-22,0) -- (-22,5) -- (-13, 13.5);
      \draw (0,0) -- (-13,13.5);
      \draw[dashed] (-16.5,13.5) -- (-9.5,13.5) node[anchor=west] {$t = \frac{5n}{4}+1$};

      \draw[->] (-13,13.5) -- (-13,15);

      \draw (-13, 0.5) -- (-13, -0.5) node[anchor=north] {$\frac{5n}{4} + 1$};
    \end{tikzpicture}
    \caption{Thermograph of $\{ \{ 3n + 2 \mid 2n + 2 \} \mid 0\}$}
    \label{fig:thermograph-2}
  \end{figure}

  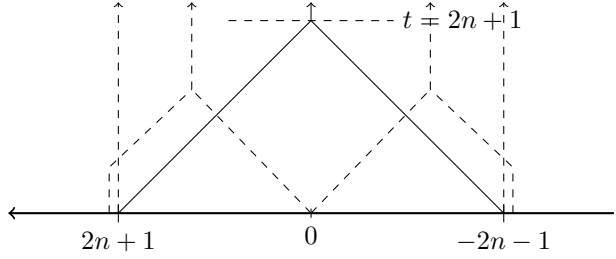
\begin{figure}[H]
    \centering
    \begin{tikzpicture}[scale=0.122]
      \draw[thick,<-] (-33,0) -- (33,0); 
      \draw[dashed] (-9,21) -- (9,21) node[anchor=west] {$t = 2n + 1$};

      \draw (0, 0.5) -- (0, -0.5) node[anchor=north] {$0$};
      \draw (-21, 1) -- (-21, -1) node[anchor=north] {$2n + 1$};
      \draw (21, 1) -- (21, -1) node[anchor=north] {$-2n - 1$};
      \draw (-21,0) -- (0,21);
      \draw (21,0) -- (0,21);
      \draw[->] (0, 21) -- (0, 23);

      \draw[dashed] (-22,0) -- (-22,5) -- (-13, 13.5);
      \draw[dashed] (0,0) -- (-13,13.5);
      \draw[dashed,->] (-13,13.5) -- (-13,23);
      \draw[dashed,->] (-21,0) -- (-21,23);

      \draw[dashed] (22,0) -- (22,5) -- (13, 13.5);
      \draw[dashed] (0,0) -- (13,13.5);
      \draw[dashed,->] (13,13.5) -- (13,23);
      \draw[dashed,->] (21,0) -- (21,23);
    \end{tikzpicture}


    \caption{Thermograph of $\pm \big\{ 2n + 1, \{ \{ 3n + 2 \mid 2n + 2 \} \mid 0\} \big\}$}
    \label{fig:thermograph}
  \end{figure}

  Therefore the temperature of the position from \cref{thm-canonical} is $2n + 1$. \qedhere
\end{proof}

Observe that the degree of the caterpillar $C(n+1,n,n+1)$ is $n+2$, thus the difference between the temperature and the degree is $n-1$.
We therefore have the following result.
\begin{theorem}
  Over all \snort positions, the difference between the temperature and the degree is unbounded.
\end{theorem}

\section{Future work}

Although we have shown that the temperature of \snort cannot be bounded by a multiple of the degree, it still seems likely that the temperature of \snort could be bounded by some polynomial in the $n$th degrees.
In particular, we give a tentative conjecture that the first and second degrees are sufficient.

We define $\deg_2(v)$ as the number of vertices at distance 2 from vertex $v$.
The second degree of a graph is the maximum second degree of its vertices, i.e.\ $\deg_2(G) = \max_{v\in V(G)} \deg_2(v)$.

\begin{conjecture}\label{conjecture}
  When playing \snort on the graph $G$, we have
  \[t(G) \leq \deg(G) + \frac{1}{2}\deg_2(G).\]
\end{conjecture}

For example, for $G=C(n+1,n,n+1)$ we have $\deg_2(G)=2n+2$ and the temperature satisfies $t(G)=2n+1< (n+1)+(n+1)=\deg(G)+\deg_2(G)/2$.

We believe that this conjecture may be correct due to the following: The first move in any position will tint all adjacent vertices.
When continuing with alternating moves, the first player will get to tint about half of the vertices at distance 2.
Thus each player can reserve $\deg(G)+\deg_2(G)/2$ many vertices for themselves.

Although all positions we have computed temperatures for satisfy this bound, we have not been able to compute temperatures where the third degree is significantly larger than the first and second degrees, which is likely important since all positions we have found where $t(G)>\deg(G)$ have a large second degree compared to the first degree.
Thus computing temperatures for such positions will likely reveal whether \cref{conjecture} is correct or whether a bound on the temperature instead would be an infinite series in the $n$th degrees.

\medskip

Another question of interest, posed by Aaron Siegel after a talk presenting these results, is whether for \snort $t(G)/\deg(G)$ is bounded.
The caterpillars we considered satisfy $t(G)/\deg(G)\leq 2$.
Thus it would be interesting to see whether this ratio can be larger.

\medskip

Finally, we have shown that $t(G)-\deg(G)$ can take on any non-negative integer value.
The temperature of a position can be any dyadic rational in $[-1,\infty)$.
Our third proposed problem is: for a given positive dyadic rational $d$, find a \snort position $G$ such that $d=t(G)-\deg(G)$.

\end{document}